\documentclass[11pt]{article}
\usepackage{amssymb}
\usepackage[perpage,symbol]{footmisc}
\usepackage{listings}
\usepackage{amsfonts}
\usepackage{enumerate}
\usepackage{amsmath}
\usepackage{cite}
\usepackage{array}
\usepackage{booktabs}
\usepackage{pgf,tikz}
\usetikzlibrary{calc,arrows}
\usepackage{latexsym,bm}
\usepackage{graphicx}
\usepackage{color}
\usepackage{amscd}
\usepackage{epsfig}
\usepackage{circuitikz}

\begin{document}

\newtheorem{theorem}{Theorem}[section]
\newtheorem{corollary}[theorem]{Corollary}
\newtheorem{definition}[theorem]{Definition}
\newtheorem{proposition}[theorem]{Proposition}
\newtheorem{lemma}[theorem]{Lemma}
\newtheorem{example}[theorem]{Example}
\newtheorem{algorithm}[theorem]{Algorithm}
\newtheorem{conjecture}[theorem]{Conjecture}
\newenvironment{proof}{\noindent {\bf Proof.}}{\rule{3mm}{3mm}\par\medskip}
\newcommand{\remark}{\medskip\par\noindent {\bf Remark.~~}}
\title{Counting Collatz Numbers}
\author{Chunlei Liu\footnote{Shanghai Jiao Tong Univ., Shanghai 200240, clliu@sjtu.edu.cn.}}
\date{}
\maketitle
\thispagestyle{empty}

\abstract{The counting function $\pi(x)$ for the numbers satisfying the Collatz conjecture is studied. A related exponential congruence equation is investigated, yielding a method to construct its solutions from free tuples, and enabling us to obtain the inequality $$\pi(x)\geq x^{0.3227},\ x\rightarrow+\infty.$$
The historical record is 0.84.
}

\noindent {\bf Key words}:  Collatz conjecture, 3x+1 problem, exponential congruence equation, ordered partition.

\section{\small{INTRODUCTION}}
\hskip .2in
We begin with the following definitions.
\begin{definition}[Syracuse map] The Syracuse map $S$ is an operator on the set of odd numbers acting
according ro the formula
$$S(n)=(3n+1)2^{-{\rm ord }_2(3n+1)}.
$$
\end{definition}
The $l$-th iteration of $S$ is denoted as $S^l$.
\begin{definition}[Collatz number] An odd number $n$ is called a Collatz number if
$S^l(n)=1$ for some
$l$.
The level of a Collatz number is the smallest number $l$ such that
$S^l(n)=1$.
\end{definition}
We now recall the Collatz conjecture.
\begin{conjecture}[Collatz conjecture]Every odd number is a Collatz number.\end{conjecture}
\begin{definition}[Collatz number counter]
Let
$$\pi(x)=\#\{n\leq x\mid n \text{ is  a Collatz number} \},\ x>0.$$
\end{definition}
In 1978, Crandall\cite{Cr78} proved that there is a positive constant $\beta>0$ such that
$$\pi(x)\geq x^{\beta},\ x\rightarrow+\infty.$$
In 1989, Krasilov \cite{Kr89} proved that
$$\pi(x)\geq x^{0.43},\ x\rightarrow+\infty.$$
In 1993, Wirsching \cite{Wi93} proved that
$$\pi(x)\geq x^{0.48},\ x\rightarrow+\infty.$$
In 1995, Applegae-Lagarias \cite{AL2} proved that
$$\pi(x)\geq x^{0.81},\ x\rightarrow+\infty.$$
In 2003, Krasilov-Lagarias \cite{KL03} proved that
$$\pi(x)\geq x^{0.84},\ x\rightarrow+\infty.$$
In 2022, Tao \cite{Tao22} studied the numbers whose orbit under the Syracuse attains almost bounded values.
In this paper, we shall prove the following result.
\begin{theorem}\label{main} There is a positive constant $c$ such that
$$\pi(x)\geq (1+o(1))x^{\frac13H_2(\frac{1}{2+\frac23\log_43})},\ x\rightarrow+\infty,$$
where  $H_2(p)=-p\log_2p-(1-p)\log_2(1-p)$.
In particular,
$$\pi(x)\geq x^{0.3227},\ x\rightarrow+\infty.$$\end{theorem}

In the next section, we shall investigate an exponential congruence equation, and  give a method to construct its solutions from free tuples. In the section after the next, we shall relate Collatz numbers to the mentioned congruence equation. In the final section we shall give a proof of the main result of this paper.
\section{Entering Congruence Equations}
In this section we relate Collatz numbers to a congruence equation.
\begin{definition}[$(2,3)$-primary Congruence Equation]
Let $x_1,\cdots,x_l$ be $l$ variables taking values in numbers. Then the congruence equation
$$
\left\{
  \begin{array}{ll}
    2^{x_1+\cdots+x_l}\equiv \sum_{j=0}^{l-1}3^j2^{x_{j+2}+\cdots+x_l}({\rm mod} 3^l), & \hbox{} \\
    2^{x_1+\cdots+x_l}\not\equiv \sum_{j=0}^{l-1}3^j2^{x_{j+2}+\cdots+x_l}({\rm mod} 3^{l+1}). & \hbox{}
  \end{array}
\right.$$
is called the $(2,3)$-primary congruence equation of level $l$.
\end{definition}
\begin{definition}
Let $(v_1,\cdots,v_l)$ be a  solution of the $(2,3)$-primary congruence equation of level $l$.
We write
$$B(v_1,\cdots,v_l))
=3^{-l}(2^{v_1+\cdots+v_l}-\sum_{j=0}^{l-1}3^j2^{v_{j+2}+\cdots+v_l}).$$
\end{definition}
\begin{lemma}Let $n>1$ be a Collatz number prime to $3$ of level $l$, and
$$v_{m}={\rm ord}_2(3S^{l-m}(n)+1),\ m=1,\cdots,l.$$
Then $v_1>2$, and
$$n=3^{-l}(2^{v_1+\cdots+v_l}-\sum_{j=0}^{l-1}3^j2^{v_{j+2}+\cdots+v_l}).$$ In particular, $(v_1,v_2,\cdots,v_l)$ is a  solution of the $(2,3)$-primary congruence equation of level $l$ such that
$$B(v_1,\cdots,v_l))=n.$$
\end{lemma}
\begin{proof}
Firstly, if $l=1$, then
$$v_{1}={\rm ord}_2(3n+1).$$
Thus
$$1= S(n)=(3n+1)2^{-v_1}.$$
Hence $v_1>2$, and
$$n=3^{-1}(2^{v_1}-1).$$
We now assume that $l>1$.
Then $$v_{m}={\rm ord}_2(3S^{l-1-m}(S(n))+1),\ m=1,\cdots,l-1.$$
By induction, $v_1>2$, and
$$S(n)=3^{-l+1}(2^{v_1+\cdots+v_{l-1}}-\sum_{j=0}^{l-2}3^j2^{v_{j+2}+\cdots+v_{l-1}}).$$
Hence
$$n=\frac{2^{v_l}S(n)-1}{3}=3^{-l}(2^{v_1+\cdots+v_l}-\sum_{j=0}^{l-1}3^j2^{v_{j+2}+\cdots+v_l}).$$
The lemma is proved.\end{proof}
\begin{lemma}
If $(v_1,v_2,\cdots,v_l)$ with $v_1>2$ is a  solution of the $(2,3)$-primary congruence equation of level $l$, and
$$n=B(v_1,\cdots,v_l)),$$
then
$n$
is a Collatz number prime to $3$  of level $l$, and
$$v_{l-m}={\rm ord}_2(3S^m(n)+1),\ m=0,\cdots,l-1.$$
\end{lemma}
By this lemma, different tuples give rise to different Collatz numbers.
\begin{proof}
Firstly, if $l=1$, then
$$1\neq n=3^{-1}(2^{v_1}-1)\not\equiv0({\rm mod} 3).$$
Thus
$$v_{1}={\rm ord}_2(3n+1),$$
and
$$ S(n)=(3n+1)2^{-v_1}=1.$$
Therefore the assertion in the lemma is true when $l=1$.
We now assume that $l>1$.
It is easy to see that
$${\rm ord}_2(3n+1)=v_l,$$
and
$$S(n)=(3n+1)2^{-v_l}=3^{-l+1}(2^{v_1+\cdots+v_{l-1}}-\sum_{j=0}^{l-1}3^j2^{v_{j+2}
\cdots+v_{l-1}}). $$
In particular, $(v_1,v_2,\cdots,v_{l-1})$ is a  solution of the $(2,3)$-primary congruence equation of level $l$.
By induction, $S(n)$
is a Collatz number  prime to $3$  of level $(l-1)$, and
$$v_{l-1-m}={\rm ord}_2(3S^{m}(S(n))+1),\ m=0,\cdots,l-2.$$
Therefore, $n$
is a Collatz number prime to $3$   of level $l$, and
$$v_{l-m}={\rm ord}_2(3S^m(n)+1),\ m=0,\cdots,l-1.$$
The lemma is proved.
\end{proof}
\section{Solving Congruence Equations}
In this section we try to solve the congruence equation introduced in the last section.
We give a method to construct its solutions from almost free tuples.
\begin{lemma}
Let $(u_1,\cdots,u_l)$ be an $l$-tuple of numbers prime to $3$. Then there is  a unique solution $(v_1,\cdots,v_l)$ of the $(2,3)$-primary congruence equation of level $l$ such that
$$[\frac{v_j+1}{6}]=u_j-1,\ j=1,\cdots,l.$$
\end{lemma}
\begin{proof}
Firstly, if $l=1$, then it is easy to see that $v_1=6u_1-4$ is  the unique solution of the $(2,3)$-primary level-$1$ congruence equation
$$2^{x_1}\equiv 1 ({\rm mod} 3),\text{ and }2^{x_1}\not\equiv 1 ({\rm mod} 3^2)$$
 such that
$$[\frac{v_1+1}{6}]=u_1-1.$$
Secondly, we assume that $l>1$ and $(v_1,\cdots,v_{l-1})$ is the unique solution of the $(2,3)$-primary level-$(l-1)$ congruence equation
$$ 2^{x_1+\cdots+x_{l-1}}\equiv\sum_{j=0}^{l-2}3^j2^{x_{j+2}+\cdots+x_{l-1}}({\rm mod} 3^{l-1})$$
 such that
$$[\frac{v_j+1}{6}]=u_j-1,\ j=1,\cdots,l-1.$$
Then $$2^{v_1+\cdots+v_{l-1}}\equiv\sum_{j=0}^{l-2}3^j2^{v_{j+2}+\cdots+v_{l-1}}({\rm mod} 3^{l-1}).$$
Let
$$a=3^{1-l}(2^{v_1+\cdots+v_{l-1}}-\sum_{j=0}^{l-2}3^j2^{v_{j+2}+\cdots+v_{l-1}}).$$
Then $a\neq0({\rm mod} 3)$. Let $r_l$ be the unique binary digit such that
$$2^{r_l}\equiv a({\rm mod} 3),$$
and let
$w_l$ be the largest ternary digit such that
$$4^{w_l}2^{r_l}a\not\equiv1({\rm mod} 3).$$
Let
$$v_l=6u_1-2w_l-r_l.$$
Then
$$2^{v_l}a\equiv4^{u_l}\equiv1({\rm mod} 3),\text{ and }2^{v_1}a\not\equiv 1 ({\rm mod} 3^2),$$
and $$[\frac{v_j+1}{6}]=u_j-1,\ j=1,\cdots,l.$$
Let
$$b=\frac{2^{v_l}a-1}{3}.$$
Then
$$2^{v_1+\cdots+v_{l}}-\sum_{j=0}^{l-1}3^j2^{v_{j+2}+\cdots+v_{l}}=3^lb
\equiv0({\rm mod} 3^{l}),$$
and
$$2^{v_1+\cdots+v_{l}}-\sum_{j=0}^{l-1}3^j2^{v_{j+2}+\cdots+v_{l}}=3^lb
\not\equiv0({\rm mod} 3^{l}),$$
That is, $(v_1,\cdots,v_l)$ a solution of the $(2,3)$-primary congruence equation of level $l$ such that
$$[\frac{v_j+1}{6}]=u_j-1,\ j=1,\cdots,l.$$
From the inductive construction process of $(v_1,\cdots,v_l)$, we see that $(v_1,\cdots,v_l)$ is uniquely determined by $(u_1,\cdots,u_l)$.
The lemma is proved.
\end{proof}
\section{Bounds}
In this section we prove the main theorem of this paper.
\begin{lemma}
Let $n,l\in\mathbb{N}$ with $\alpha n\leq l\leq n$, where $\alpha$ is a positive constant. Then
$$\log_2{n\choose l}=H_2(l/n)n+O(1),\ n\rightarrow+\infty.$$
\end{lemma}
\begin{proof}
This follows from the last lemma and Stirling's formula.\end{proof}
\begin{definition}[Fixed level Collatz number counter]
For $l\in\mathbb{N}$, let
$$\pi(x,l)=\#\{n\leq x\mid n \text{ is  a Collatz number of level } l\},\ x>0.$$
\end{definition}
\begin{definition}[Ordered partition counter]
For $l\in\mathbb{N}, y>0$, let
$$\omega(y,l)=\#\{(u_1,\cdots,u_l)\in(\mathbb N+1)^l\mid \sum_{j=1}^lu_j\leq y \}.$$
\end{definition}
\begin{lemma}[Partition Bound] Let $x>0$. Then
$$\pi(x,l)\geq\omega(\frac13\log_4x+l(1+\frac13\log_43),l).$$
\end{lemma}
\begin{proof}
In fact, let $(u_1,\cdots,u_l)$  be an $l$-tuple of numbers with $u_1>1$ such that  $$\sum_{j=1}^lu_j\leq \frac13\log_4x+l(1+\frac13\log_43).$$ That is, $(u_1,\cdots,u_l)$  is an ordered $l$-partition of a number   $\leq \frac13\log_4x+l(1+\frac13\log_43).$
 Let $(v_1,\cdots,v_l)$ be the unique solution of the $(2,3)$-primary congruence equation of level $l$
such that
$$[\frac{v_j+1}{6}]=u_j-1,\ j=1,\cdots,l.$$
Then  $B(v_1,\cdots,v_l)$ is a Collatz number of level $l$.
It is easy to see that
$$B(v_1,\cdots,v_l)\leq 3^{-l}2^{v_1+\cdots+v_l}\leq (192)^{-l}4^{3(u_1+\cdots+u_l)}\leq x.$$
It is also easy to see that different ordered partitions give rise to different Collatz numbers. It follows that
$$\pi(x,l)\geq\omega(\frac13\log_4x+l(1+\frac13\log_43),l).$$
The lemma is proved.\end{proof}

{\it Proof of Theorem \ref{main}.}
Let
$l=[\frac{\log_4x}{3+1\log_43}]$, and $n=[\frac13\log_4x+l(1+\frac13\log_43)]$.
Then
\begin{eqnarray*}
      \pi(x) &\geq& \pi(x,l). \\
       &\geq&\omega(\frac13\log_4x,l)\\
&\geq&{n\choose l}\\
&\geq& (1+o(1))x^{\frac13H_2(\frac{1}{2+\frac23\log_43})}.
    \end{eqnarray*}
This completes the proof of Theorem \ref{main}.


\begin{thebibliography}{200}
\bibitem[AL1]{AL1}Applegate, D. and Lagarias, J.C., Density bounds for the $3x+1$ problem I. Tree-search method,  Math. Comp. 64(1995), pp. 411-426.
\bibitem[AL2]{AL2}Applegate, D. and Lagarias, J.C., Density bounds for the $3x+1$ problem II. Krasikov method,  Math. Comp. 64(1995), pp. 427-438.
\bibitem[Cr78]{Cr78}Crandall, R. E., On the $3x+1$ problem, Math. Comp. 32 (1978), pp. 1281-1292.
\bibitem[Kr89]{Kr89}Krasikov, I., How many numbers satify the $3X+1$ conjecture?, Internat. J. Math. \& Math. Sci. 12(1989), pp. 791-796.
\bibitem[KL03]{KL03}Krasikov, I and Lagarias, J.C., Bounds for the $3x+1$ problem using difference inequalities,  Acta Arith. 109(2003), pp. 237-258.
\bibitem[La85]{La85}Lagarias, J.C., The $3x+1$ problem and its generalizations,  Amer. Math. Monthly 92(1985), pp. 3-21.
\bibitem[Tao22]{Tao22}Tao, T., Almost all orbits of the Collatz map attain almost bounded values,  Forum Math., Pi, 2022-01, vol. 10, e12.
\bibitem[Wi93]{Wi93}Wirsching, G. J., An improved estimate concerning $3N+1$ predecessor sets, Acta Arith. 63(1993), pp. 205-210.
\bibitem[Wi98]{Wi98}Wirsching, G. J., The Dynamical System Generated by  the $3n+1$ Function, Lecture notes in Math. No. 1631, Springer-Verlag: Berlin 1998.






\end{thebibliography}
\end{document}